\documentclass{amsart}

\usepackage{eufrak}

\usepackage{amssymb,amsmath, amscd, mathrsfs, yfonts, bbm, graphics, dsfont}

\title[H${}^2$ Spaces of Non-Commutative Functions] {H${}^2$ Spaces of Non-Commutative Functions}

\author{Mihai Popa and Victor Vinnikov}
\address{Department of Mathematics, University of Texas at San Antonio, One UTSA Circle
San Antonio, Texas 78249, USA,}
\address{
Institute of Mathematics `Simion Stoilow' of the Romanian Academy, P.O. Box 1-764,
Bucharest, RO-70700, Romania}
\email{mihai.popa@utsa.edu}
\address{Department of Mathematics, Ben Gurion University of Negev, Be'er Sheva
84105, Israel}
\email{vinnikov@cs.bgu.ac.il}

\newtheorem{claim}{}[section]
\newtheorem{defn}[claim]{Definition}
\newtheorem{thm}[claim]{Theorem}

\newtheorem{remark}[claim]{Remark}
\newtheorem{prop}[claim]{Proposition}
\newtheorem{cor}[claim]{Corollary}

\newcommand{\Tr}{\textrm{Tr}}

\newcommand{\ncspace}[1]{\ensuremath{#1}_{\text{nc}}}

\newcommand{\cA}{\mathcal{A}}

\newcommand{\cF}{\mathcal{F}}
\newcommand{\cFml}{\mathcal{F}_m^{ [ l ] } }

\newcommand{\dBn}{\left(\partial_S(\mathbb{B}^m)_{\text{nc}}\right)_n }
\newcommand{\dDN}{\left(\partial_S(\mathbb{D}^m)_{\text{nc}}\right)_N}
\newcommand{\dBN}{\left(\partial_S(\mathbb{B}^m)_{\text{nc}}\right)_N }

\newcommand{\cK}{\mathcal{K}}

\newcommand{\col}{\|_{\text{col}}}
\newcommand{\row}{\|_{\text{row}}}

\newcommand{\lra}{\longrightarrow}

\usepackage{amssymb}
\input xy
\xyoption{all}

\newcommand{\cV}{\mathcal{V}}
\newcommand{\cW}{\mathcal{W}}

\newcommand{\ED}{E_{\mathbb{D}^m}}

\begin{document}

\maketitle

\begin{abstract}
We define the Hardy spaces of free noncommutative functions on the noncommutative polydisc
and the noncommutative ball and study their basic properties.
Our technique combines the general methods of noncommutative function theory
and asymptotic formulae for integration over the unitary group.
The results are the first step in developing the general theory of free noncommutative bounded symmetric domains
on the one hand
and in studying the asymptotic free noncommutative analogues of classical spaces of analytic functions on the other.  
\end{abstract}

\section{Introduction}

\let\thefootnote\relax\footnotetext{
This work was partially supported by a grant of the Romanian National Authority for Scientific Research, CNCS – UEFISCDI, project number PN-II-ID-PCE-2011-3-0119 and Simmons Foundation Grant No. 360242.}

There emerged, over the years, a general paradigm of passing from the commutative setting
to the free noncommutative setting: we replace a vector space by the disjoint union of square matrices of all sizes
over this vector space. 
Some instances of this paradigm are Amitsur's theory of rational identities \cite[Chapter 8]{row80},
operator space theory \cite{effros-ruan,pisier}, 
free probability --- when viewed asymptotically \cite{vdn,speichernica}, 
and free noncommutative algebraic and semialgebraic geometry \cite{hkmcc-fcag}.
Another instance is free noncommutative function theory
that originated with the work of J. L. Taylor on noncommutative spectral theory
\cite{t1,t2} and was further developed by Voiculescu \cite{Voi04,Voi09}
and Kaliuzhnyi-Verbovetskyi--Vinnikov \cite{vv-dkv}
(we refer to \cite{vv-dkv} for a historical account and further references).

The aim of this paper is to take first steps towards the theory of Hardy spaces on noncommutative domains, 
within the framework of noncommutative function theory.
More specifically we introduce and study the Hardy space $H^2$ of noncommutative functions
on the noncommutative polydisc
\[
( \mathbb{D}^m )_{ \text{nc} } 
=
\coprod_{ n =1 }^\infty \{ ( X_1, \dots, X_m ) \in ( \mathbb{C}^{ n \times n } )^m : \| X_j \| < 1,\  j = 1, \dots, m \}
\]
and the noncommutative ball
\[
(  \mathbb{B}^m )_{ \text{nc} }
=
 \coprod_{ n =1}^\infty
\{ ( X_1, \dots, X_m ) \in ( \mathbb{C}^{ n \times n } )^m : \sum_{ i =1}^m X_i^\ast X_i <  I_n  \}.
\]
Our technique combines the methods of noncommutative function theory, especially the Taylor--Taylor noncommutative
power series expansions, see \cite[Chapter 7]{vv-dkv},
and asymptotic formulae for integration over the unitary group coming from random matrix
theory and free probability, see \cite{voi1,speicher-mingo} and \cite{collins,collins-sniady}.
The resulting Hardy spaces have some of the same basic properties as their commutative counterparts;
the most striking difference is that they are {\em not complete}, though their completions
can be identified as Hilbert spaces of noncommutative functions 
(in fact, noncommutative reproducing kernel Hilbert spaces) on a certain noncommutative set.
This is related to the different (as opposed to the commutative case) convergence patterns
for noncommutative power series, see \cite[Section 8.3]{vv-dkv}.
A general theory of noncommutative reproducing kernel Hilbert spaces and their multipliers
is developed in the forthcoming paper \cite{bmv}.

The paper is organized as follows.
Section 2 discusses the preliminaries on the asymptotic integration over the unitary group,
noncommutative function theory,
and the noncommutative unit balls of two operator space structures on ${\mathbb C}^m$,
namely the noncommutative polydisc and the noncommutative ball mentioned above,
their distiniguished boundaries and the invariant measures thereupon.
Section 3 contains the main results on the definition and the basic properties of the Hardy spaces.

In the case of bounded noncommutative functions, asymptotic integral formulae in terms
of tracial integrals over the distinguished boundary with respect to the invariant measure
have been obtained by Voiculescu \cite[Chapters 14--16]{Voi09}
for the noncommutative polydisc and the square noncommutative matrix ball,
i.e., the noncommutative unit ball of the noncommutative space over ${\mathbb C}^{m \times m}$
with its obvious operator space structure. 

The noncommutative polydisc and the noncommutative ball are of course noncommutative analogues
of the usual polydisc and ball in ${\mathbb C}^m$.
In a forthcoming paper we will consider noncommutative analogues of other matrix balls
and of irreducible bounded symmetric domains of type $II$ and $III$,
and develop a general theory of noncommutative Jordan triples.
It would be also interesting to consider in this context 
the unit ball in the $OH$ operator space norm on ${\mathbb C}^m$ \cite{pisier}.
In a different direction, the results here point towards a study of 
asymptotically defined spaces of noncommutative functions as analogues of various classical spaces of analytic functions.
In particular, it would be interesting to study the noncommutative Bargmann--Fock space 
in relation to free stochastic processes \cite{ajs}.


\section{Preliminaries}


\subsection{Haar Unitaries and Free Independence}\label{haar}

Let $ N $ be a positive integer and $ \mathcal{U}(N) $ be the compact group of the $ N \times N $ unitary matrices with complex entries. The Haar measure on $ \mathcal{U} (  N ) $ will be denoted with $ d\mathcal{U}_N $.

 For each $ i, j \in \{1, 2, \dots, N \}$ we define the maps $ u_{ i, j }: \mathcal{U}( N ) \lra \mathbb{C} $ giving the $ i, j $-th entry of each element from $ \mathcal{ U}( N ) $. As shown in \cite{collins}, the maps $ u_{ i, j } $ are in  
$  L^\infty( \mathcal{U}(N), d\mathcal{U}_N ) $. 
Let $ S_n $ be the symmetric group of order  $ n $; for  $ \sigma \in S_n $  denote by 
$ \#(\sigma ) $  the number of cycles in a minimal decomposition of the permutation $\sigma $.
The following result is shown in \cite[Corollary 2.4]{collins-sniady}:

\begin{thm}\label{unitary:wg}
There exists a map $ \displaystyle  \text{Wg}: \mathbb{Z}_{ + } \times \cup_{ n=1}^\infty S_n \lra \mathbb{R } $   such that:
\begin{enumerate}
\item[(1)]The function $ \text{Wg} ( \cdot, \cdot ) $ is analytic at $ \infty $ in the first variable and, 
for each $ \sigma \in S_n $, the limit 
$ \displaystyle \lim_{N \rightarrow\infty}\frac{ \text{Wg}( N, \sigma ) }{ N^{2n-\#(\sigma ) } } $ 
exists and is finite.
\item[(2)] For any indices $i_k, i^\prime_k, j_k, j^\prime_k \in \{ 1, 2, \dots, N \} $, 
where $ 1 \leq k \leq n $, 
we have that 
\[
\int_{\mathcal{U}(N )}
  u_{i_1,j_1}\cdots u_{i_n, j_n}
 \overline{u_{i^\prime_1, j^\prime_1}}\cdots 
\overline{u_{i^\prime_n, j^\prime_n}} 
d\mathcal{U}_N 
=\sum_{\sigma, \tau\in S_n} \left( 
 \text{Wg}( N, \tau\sigma^{-1} ) 
\cdot
\prod_{ k =1}^n \delta_{i^{}_k, i^\prime_{ \sigma( k ) } }
\delta_ { j^{}_k , j^\prime_{ \tau( k ) } }  
\right)
\]
\end{enumerate}
Moreover, if $ m \neq n $ , then 
\[
\int_{\mathcal{U}(N ) }
 u_{i_1,j_1}\cdots u_{i_n, j_n}\overline{u_{i^\prime_1, j^\prime_1}}\cdots 
  \overline{u_{i^\prime_m, j^\prime_m}} 
    d\mathcal{U}_N
=
 0. 
\]
\end{thm}

An immediate consequence of the result above is the following:

\begin{cor}\label{rem:tr}
 Let $ U: \mathcal{U}(N) \lra \mathbb{C}^{N \times N }$, $ U = [ u_{i, j}]_{i, j=1}^N $. 
Then, for all non-zero integers $\alpha$, 
\[
\int_{ \mathcal{U}(N ) } \textnormal{Tr} ( U^\alpha ) d\mathcal{U}_N = 0,
\]
where $ \textnormal{Tr} $ denotes the non-normalized trace.
\end{cor}
\begin{proof}
 Suppose  that $ \alpha > 0 $. Then
\[
 \int_{ \mathcal{U}(N ) } \Tr( U^\alpha ) d\mathcal{U}_N =
\sum_{ 1 \leq i_1, \dots, i_\alpha \leq N } 
\int_{ \mathcal{U}(N ) } u_{ i_1, i_2} \cdots u_{ i_{ \alpha - 1 }, i_{\alpha } } u_{ i_{\alpha}, i_1 } d\mathcal{U}_N.
\]
From the last part of Theorem \ref{unitary:wg}, all the terms in the above summation are zero, hence the conclusion. 
Since $ U^{ -1} = U^\ast $, the case $ \alpha < 0 $ is similar.
\end{proof}

When studying the joint asymptotic behavior of several large random matrices with independent entries, 
an important tool is the notion of \textbf{free independence} (see, for example, \cite{voi1}, \cite{mingo-popa}). 
As shown in the extensive literature on the subject (see \cite{vdn}, \cite{speichernica}, \cite{mingo-popa},  \cite{mingo-popa-t}),  
this is in fact the natural relation of independence in a non-commutative framework.   
The precise definition for the version  of  the notion free independence 
that will be used in the present work is presented below.
Suppose  that $\mathcal{A}$ is a unital C$^\ast$-algebra and $ \phi: \mathcal{ A } \lra \mathbb{C} $ 
is a positive conditional expectation. 
A family $ \{ \mathcal{A}_j\}_{j \in J } $ of  unital C$^\ast$-subalgebras  of $ \mathcal{A} $ is 
said to be \textbf{free} if any alternating product of centered (with respect to $\phi $) 
elements from $ \{ \mathcal{A}_j\}_{j \in J } $  is centered, i.e., 
for any $ n > 0 $, any $\epsilon(k) \in J $ 
($ 1\leq k \leq n $) such that 
$\epsilon(k) \neq \epsilon( k + 1 ) $ and any $ a_k \in \mathcal{A}_{ \epsilon( k ) }$ 
such that $\phi(a_k) = 0$ we have that 
$\phi(a_1 a_2 \cdots a_n ) = 0 $.  
Subsets  $  M_1, M_2, \dots, M_n $  of $ \mathcal{A} $ are said to be free or free independent 
if the unital C$^\ast$-algebras generated by the elements of each of them form a free family.

The following result, Theorem \ref{thm:free},  
is proved in \cite{voi1} and, in a more general framework,  in \cite{collins-sniady}, \cite{speicher-mingo}.
We first introduce some notation.
Let $\mathfrak{A} = \{ A_{ j, N }\}_{ j\in J, N \geq 1 } $
 be an ensemble of matrices such that
 $ A_{j, N }\in \mathbb{C }^{ N \times N } $ for all $ j \in J $. 
The ensemble $ \mathfrak{ A } $ is said to have 
limit distribution if for any
 $ m \in \mathbb{Z}_{ + } $ and $ j_1, \dots, j_m \in J $ 
 the limit
$ \displaystyle   \lim_{ N \rightarrow \infty }\frac{1 }{ N } \Tr ( A_{ j_1, N }\cdots A_{ j_m, N } ) $
 exists and is finite. 
Also, let $ j_1, \dots, j_s \in J $;
a polynomial in $ s $ non-commutative variables, 
$ p \in \mathbb{C}\langle x_{j_1}, \dots, x_{j_s} \rangle $
is said to be asymptotically centered in $ \mathfrak{A} $  if
 $ \displaystyle \lim_{ N \rightarrow \infty } \frac{1}{N} \Tr ( p( A_{ j_1, N}, \dots, A_{ j_s, N } ) ) = 0 $.

\begin{thm}\label{thm:free}
Let $  m $ be a positive  integer; for $ 1 \leq k \leq m $  and $ 1 \leq i, j \leq N $ consider the random variables 
$ u_{ i, j }^{ ( k )}: \mathcal{ U }( N ) \lra \mathbb{C} $ such that
 $ u_{ i, j }^{ ( k )} $  and $  u_{ i,j } $  are identically distributed for each $ i, j , k $ and  
 $ \{ u_{i, j }^{ ( k )} \}_{ i, j = 1 }^N $ 
 are independent.
Finally, for each $k $ and $ N ,$  consider the matrix  
$ U_{ k, N}  \in  L^\infty( \mathcal{U}(N), d\mathcal{U}_N )^{ N \times N } $,
 having the entries $ u_{ i, j}^{ ( k )} $. 

Suppose that  
$ \mathfrak{A} = \{ A_{ j, N } \}_{j \in J, N \geq 1 } $
 is an ensemble of complex matrices that has limit distribution.
 Then the ensembles  of random matrices
 $\{ U_{1, N}, U_{ 1, N }^\ast \}_{ N \geq 1 } $, \\
$\{ U_{2, N}, U_{ 2, N }^\ast \}_{ N \geq 1 } $, \dots,
$\{ U_{ m, N}, U_{ m , N }^\ast \}_{ N \geq 1 } $ 
and
$ \mathfrak{A} $
are asymptotically free with respect to the functional
 $  \displaystyle
 \int_{ \mathcal{U} ( N ) } \frac{1}{N} \text{Tr} ( \cdot ) d\mathcal{U}_N  $, in the following sense:
\[
\lim_{ N \rightarrow \infty } \int_{ \mathcal{U} ( N ) } \frac{1}{N}
 \text{Tr} ( p_1 \cdot p_2 \cdots p_k ) d\mathcal{U}_N
= 0 
\]
for any $ p_1, \dots, p_k $ either centered polynomials in $ \{ U_{ l, N}, U^\ast_{l, N } \}_{ N \geq 1 } $ for some  $ l $ , or asymptotically centered polynomials in elements of $ \mathfrak{A} $, such that $ p_s $ and $ p_{s + 1} $ are polynomials in elements of different ensembles for all $ s $.
\end{thm}

\begin{remark}
\label{remark:new1}
In the framework of Theorem  \ref{thm:free} above, suppose that the
joint distribution of elements of $\mathfrak{A} $ does not depend on $ N $, i.e. for any
 $ j_1, j_2, \dots, j_s $ there exists a complex constant $ c(j_1, j_2, \dots, j_s)$ such that 
 \[
 \frac{1}{N} \Tr ( A_{ j_1, N} A_{ j_2, N} \cdots A_{ j_s, N } )  = c(j_1, j_2, \dots, j_s).
 \]
 
 Then 
\[ \lim_{ N \rightarrow \infty } \int_{ \mathcal{U} ( N ) } 
  \text{Tr} ( p_1 \cdot p_2 \cdots p_k ) d\mathcal{U}_N
 < \infty.
 \]
\end{remark}
\begin{proof}
First note that the condition on $\mathfrak{A} $ implies that the polynomials $p_1, p_2, \dots, p_k $ in elements of $ \mathfrak{A} $ are centered for any $ N $. Then using the analyticity in the first variable at $\infty$ of the function $\text{Wg}(\cdot, \cdot) $ from  Theorem \ref{unitary:wg}(i), it follows that the expression 
$  \displaystyle 
\int_{ \mathcal{U} ( N ) } \frac{1}{N}
 \text{Tr} ( p_1 \cdot p_2 \cdots p_k ) d\mathcal{U}_N $
 expands at $\infty $ as a Laurent series in $N $, with coefficients depending on the polynomials $ p_1, p_2, \dots, p_k $. Theorem \ref{thm:free} from above implies that all the coefficients of monomials $N^q $ with $ q \geq 0$ are null, hence the conclusion.
\end{proof}

Throughout the paper, $ \cF_m $ will denote the free monoid with $ m $ generators 
$ \{ 1, \dots, m \} $.
 The elements of $ \cF_m $ are arbitrary  words $ w = {  w_1 } \cdots { w_{ l -1 } }  {w_l } $; the length 
of the word $ w $ will be denoted by $ | w | = l $.
 We will also use the notation $ \cFml $ 
for the set of all  words from $ \cF_m $ of length $ l $.

In the next section we will utilize the following consequences of the Theorems \ref{unitary:wg} 
and \ref{thm:free} above:

\begin{cor}\label{cor:24}
 For $ y = y_1 \cdots y_{ t -1 } y_t $ a word in $ \cF_m $ 
and  $ U_{ 1, N }, \dots, U_{ m, N } $ as in Theorem \ref{thm:free},
we will denote 
$ U_N^y = U_{ y_1, N } \cdots U_{ y_{ t-1} ,  N } U_{ y_t, N}. $
With this notation, for any $ w, v \in \cF_m $ we have that:

\begin{enumerate}
\item[(i)]if $ | v | \neq | w | $, then
   $ \displaystyle
 \int_{ \mathcal{U}( N ) } \Tr\left( ( U^w_N)^\ast U^v_N \right)  d\mathcal{U}_N = 0
$,   for any positive integer
 $ N $;
 \item[(ii)] if $ | v | = | w | $ but 
 $ v \neq w $, then
  $ \displaystyle
\lim_{N \rightarrow \infty} \int_{ \mathcal{U}( N ) } 
\frac{1} {\sqrt{ N}} \Tr
\left( ( U^w_N)^\ast U^v_N \right) d\mathcal{U}_N  = 0 $;
\item[(iii)]if $ v = w $, then 
$ \displaystyle
\lim_{N \rightarrow \infty} \int_{ \mathcal{U}( N ) } 
\frac{1} { N} \Tr
\left( ( U^w_N)^\ast U^v_N \right) d\mathcal{U}_N  = 1 $
\end{enumerate}
\end{cor}

\begin{proof}
Suppose that $ w =  w_1 \cdots w_{ t - 1 } w_t $ and $ v = {v_1} \cdots v_{ s -1 } v_s $ 
(here $ | w | = t $ and $ | v | = s $). For part  (i), let 
\[ 
 L = \{
 l = (  l_{ -| v | }, l_{ -| v | +1 }, \dots, l_0, l_1, \dots, l_{ | w |} ) \in \{ 1, \dots, N \}^{ | v | + | w | +1 } : 
  l_{ - | v | } = l_{ | w | }
 \}
\]
and  $ V_k = \{ j \in \{ 1, \dots, | v | \}: v_j = k \}$, respectively 
$ W_k = \{ j \in \{ 1, \dots, | w | \} : w_j = k \} $.
Then,  from the independence of the families $ \{ u_{i, j }^{ ( k )} \}_{ i, j = 1 }^N $,
\begin{align*}
\int_{ \mathcal{U}( N ) } \text{Tr}\left( ( U^w_N)^\ast   U^v_N \right)
 &  d\mathcal{U}_N
 = \sum_{ l \in L } \int_{ \mathcal{U} ( N ) } 
\prod_{ k =1 }^{ | w | } \overline{  u^{ ( w_k ) }_{ l_{ k-1 } , l_{ k } } }
\cdot \prod_{ k =1 }^{ | v | } u^{ ( v_k ) }_{ l_{- k+1 } , l_{-  k } }
d\mathcal{U}_N\\
=& \sum_{ l \in L } \prod_{ r =1}^m ( \int_{\mathcal{U}( N ) }  
\prod_{ k\in W_r } \overline{  u^{ ( w_k ) }_{ l_{ k-1 } , l_{ k } } } \cdot
\prod_{ k \in V_r } u^{ ( v_k ) }_{ l_{- k+1 } , l_{-  k } }
d\mathcal{U}_N ).
\end{align*}
From Theorem \ref{unitary:wg}, if $\text{card}( W_r ) \neq \text{card}(V_r)$, 
then the coresponding factor in the above product vanishes, hence the conclusion.

   For (ii),  it suffices to consider the case when  $ v_1 \neq w_1 $ (since
 $ U_{k, N}^\ast U_{k, N } = \text{Id}_N $). From Corollary \ref{rem:tr}, 
$ \displaystyle 
 \int_{\mathcal{U}( N ) } \frac{ 1 } { N } \text{Tr } ( [ U_{ k, N } ]^p ) d\mathcal{U}_N = 0,
$
 for all integers $ p $, all $ N >1 $ and  all $ 1 \leq k \leq m $, and 
the conclusion follows now  from Remark \ref{remark:new1}.

Finally, for (iii),  if $ w = v $, then $ \left( U_N^w\right)^\ast U_N^v = \text{Id}_N $, and the assertion is trivial.
\end{proof}

An analogous result for the matricial block entries of a Haar unitary is given below.

\begin{cor}\label{cor:2}
Fix $ m $ a positive integer and suppose that 
$ U = [ u_{ i, j } ]_{ i, j = 1}^{mN} $, 
with the functions $ u_{ i, j }: \mathcal{U}(mN) \lra \mathbb{C} $ as defined above.
For $ 1 \leq k \leq m $ , consider 
$ U_k \in L^\infty( \mathcal{U}(mN), d\mathcal{U}_{ m N } )^{N \times N } $  
given by
 $ U_k =[ u_{ i, ( k - 1 ) N +j } ]_{ i, j = 1}^N $.
(I. e., $ U_1, \dots, U_m $ are the $ N \times N $ matricial block entries of the first $ N \times m N $ matricial row of $ U $).

 Let  $ v, w \in \cF_m $ , and, for 
 $ y =  y_ 1 \cdots y_{ s -1} y_s \in \cF_m $,   denote 
 $ U^y = U_{ y_1 } \cdots U_{ y_{ s -1} } U_{ y_s } $.
 Then:
\begin{enumerate}
\item[(i)]If $ | v | \neq | w | $, $\displaystyle \int_{ \mathcal{ U }( mN ) } \text{Tr}\left( ( U ^w)^\ast U^v
\right)d\mathcal{U}_{mN} = 0 $ for any positive integer $ N $;
\item[(ii)]If $ | v | = | w |$, but $ v \neq w $, then 
$ \displaystyle \lim_{ N \rightarrow \infty } 
\int_{ \mathcal{ U }( mN ) } \frac{ 1 } {\sqrt{ N } }
\text{Tr}\left( ( U ^w)^\ast U^v \right)
  d\mathcal{U}_{mN} = 
   0$;
\item[(iii)]If $ v  =  w  $, then 
$ \displaystyle \lim_{ N \rightarrow \infty } 
\int_{ \mathcal{ U }( mN ) } \frac{ 1 } { N }
\text{Tr}\left( ( U ^w)^\ast U^v \right)
  d\mathcal{U}_{mN} = 
   \frac{1}{m^{| v |} } $.
\end{enumerate}

\end{cor}

A similar statement holds for the first $mN \times N$ matricial column of $U$.

\begin{proof}
Parts (i), respectively (iii) are immediate consequences of  Theorem \ref{unitary:wg}, respectively of the identity $ U_k^\ast U_k =\text{Id}_N $.

 For part (ii), let us suppose that 
$ w =  { w_1 } \cdots { w_{ t - 1 } } w_t $ and $ v = {v_1} \cdots  { v_{ s - 1 } } { v_s } $, with $ w_1 \neq v_1 $.
Let $ e_{i, j} $ be the $ m \times m $ matrix with the $ i, j $ entry 1 and all other entries 0 and 
$ E_{i, j} =  e_{i, j} \otimes \operatorname{Id}_N \in \mathbb{C}^{ mN \times mN }$.
 Then for all $ 1\leq k \leq m $, we have that 
 $ \widetilde{U_k} = e_{ 1, 1 } \otimes U_k = E_{1, 1 } U E_{ k, 1 } $, hence
\begin{align}
\operatorname{Tr}\left( ( U ^w)^\ast U^v \right)& = \text{Tr} \left(
\widetilde{U_{ w_t }}^{\ast} \cdots \widetilde{U_{ w_1 }}^{\ast} \widetilde{U_{v_1}  }\cdots \widetilde{U_{v_s}  }
\right)\\ \nonumber
& \hspace{-1cm} =\text{Tr}
\left(
E_{ 1, w_t } U ^\ast E_{ 1, w_{t-1}} U^\ast\cdots E_{1, w_1}
 U^\ast E_{1, 1} U E_{v_1, 1}U \cdots E_{v_{s-1}, 1} U E_{v_s, 1}
\right).\label{eq:1}
\end{align}
To simplify the notation, we shall write
\begin{equation}\label{eq:2}
E^0_{i,j} = E_{i, j}- \delta_{i, j} \frac{1}{m} \text{Id}_{mN}.
\end{equation}
Note that $ \text{Tr}( E^0_{i, j } ) = 0  $, and that the ensemble $\{ e_{i, j}\otimes \text{Id}_N \} $ has the property from Remark \ref{remark:new1}, therefore  for all non-zero integers
 $ \alpha_0, \dots, \alpha_{n} $ and all indices
$ i, j, k, l, k_r, l_r \in \{ 1, \dots, m \}$ we have that
\begin{equation}\label{eq:3}
\lim_{ N\rightarrow \infty} \frac{1}{\sqrt{N}} \int_{\mathcal{U}(mN)} \text{Tr}
\left(
E_{i, j} U^{\alpha_0} E^0_{k_1, l_1}
 U^{\alpha_1} E^0_{k_2, l_2}\cdots E^0_{ k_{n}, l_{n} } U^{\alpha_n} E_{k, l}
\right) d\mathcal{U}_{mN} = 0,
\end{equation}
hence
 \begin{equation}\label{eq:4}
\lim_{ N\rightarrow \infty} \frac{1}{m\sqrt{N}} \int_{\mathcal{U}(mN)}
\text{Tr}
\left(
E_{ 1, w_t } U ^\ast \cdots E_{1, w_1}
 U^\ast E_{1, 1}^0
 U E_{v_1, 1}U \cdots  E_{v_s, 1}
\right)d\mathcal{U}_{mN} = 0,
\end{equation}
because, using (\ref{eq:2}) for $ E_{1, w_1}, \dots, E_{1, w_{t-1}} $ and $ E_{v_1, 1}, \dots, E_{v_{s-1}, 1} $, 
the integrand from (\ref{eq:4}) is a finite linear combination of integrands from (\ref{eq:3}).
\end{proof}


\subsection{Non-Commutative Functions and Taylor--Taylor Expansions}\label{section:22}

We define non-commutative functions 
following \cite{vv-dkv}, see also \cite{vv-dkv2} and \cite{mp-vv}.

For $ \cV $ a (complex) linear space, we will denote by $\ncspace\cV $  the set
$ \coprod_{ n = 1}^\infty  \cV^{n \times n } $. 
For a subset $ \Omega $  of  $ \ncspace\cV $, we denote
$\Omega_n = \Omega \cap \cV^{n \times n }$; 
$\Omega$ is said to be a \emph{non-commutative set}
if for all positive integers $ m, n $ and all
$ X \in \Omega_m$ and $ Y \in \Omega_m$ we have that 
$ X \oplus Y \in \Omega_{m+n}$, where $ X \oplus Y $ is the block diagonal matrix from 
$ \cV^{ ( m + n ) \times ( m + n ) } $ with 
$ X $ and $ Y $ the block entries of the main diagonal and all other entries zero.

If $ \cV $ and $ \cW $ are two linear spaces and $ \Omega $ a non-commutative subset of $ \ncspace \cV$, a mapping
$ f : \Omega \lra \ncspace \cW $
is said to a \emph{non-commutative function}  if it satisfies the following conditions:
\begin{enumerate}
\item[$\bullet$]$ f ( \Omega_n) \subset \cW^{ n \times n } $ \  for all positive integers $ n $;
\item[$\bullet$]$ f ( X \oplus Y ) = f(X) \oplus f(Y) $ \ for all $ X, Y \in \Omega $;
\item[$\bullet$]if $ X \in \Omega_n $ and
 $ T \in \mathbb{C}^{n \times n } $ is invertible with 
$  T X T^{-1} \in \Omega $, then
 \[
 f ( T X T^{ -1} ) = T f( X ) T^{-1}.
\]
\end{enumerate}

Non-commutative  functions have strong regularity properties --- for an introduction to the basic theory
see \cite{vv-dkv}. Below we will mention only a particular form of 
the Taylor--Taylor expansion, as established in
\cite[Chapter 7]{vv-dkv},  that will be extensively utilized in Section 3 of the present work.

Let $ \cV $ be a finite dimensional vector space with basis $ e_1, \dots, e_d $. For $ X \in \cV^{ N \times N }  $,
 there exist unique
 $ X_1, \dots, X_d \in \mathbb{C}^{ N \times N } $ 
such that 
$ X = X_1 e_1 + \ldots + X_d e _d .$
  If  $ w = w_1 \cdots w_t \in \cF_d $, we write $ X^w = X_{ { i_1 } } \cdots X_{ { i_t} } $.

Suppose that $ \Omega\subseteq \cV_{\text{nc} }  $
 is a non-commutative set  such that for
all $ N $, the set $ \Omega_N = \Omega \cap \cV^{ N \times N } $ 
is open, 
let $ \cW $ be a Banach space,
and suppose that 
$ f : \Omega \lra \cW_{ \text{nc} } $ 
is a non-commutative function
locally bounded on slices separately in every matrix dimension,
that is 
   for all positive integers $ N $, all $X \in \Omega_N$,
and all $ Y \in \cV^{ N \times N } $,
there exists $ \varepsilon > 0 $ such that 
the function
 $ t \mapsto f ( X + t Y ) $ is bounded for $  | t | < \varepsilon .$
Let $ b \in \Omega_1$, and 
for $ N $ a positive integer, define the set
\[
\Upsilon (\operatorname{Id}_N b)=\{ X\in \Omega_N : \  \operatorname{Id}_N b + t ( X - \operatorname{Id}_N \cdot b ) \in \Omega_N\  \text{for all } 
t \in \mathbb{C} \ \text{such that } | t | \leq 1 \}
\]
(this is the maximal subset of $\Omega_N$ that is complete circular around $\operatorname{Id}_N b$).
Then 
\cite[Theorem 7.2]{vv-dkv} (see also Theorems 7.8 and 7.10 there)
states that for 
all $ X \in \Upsilon (\operatorname{Id}_N b) $ 
\begin{equation}\label{eq:nc1}
f(X)=\sum_{ l = 0 }^\infty \left( \sum_{ | w | = l } ( X -  \operatorname{Id}_N b )^{ w } \otimes f_{ w }  \right),
\end{equation}
where the series converges absolutely and uniformly (in fact, normally) 
on compacta of  $ \Upsilon (\operatorname{Id}_N b) .$
(The Taylor--Taylor coefficients $f_{ w } \in \cW$ are given by
$f_w = \Delta_R^{w^\top} f (b,\ldots,b)$, 
where $\Delta_R^{w^\top}$ is the higher order partial difference--differential operator corresponding 
to $w \in \cF_d$.)


\subsection{Operator space structures on $ \mathbb{C}^m $}
 An operator space structure on a linear space $ \cV $ is given (see \cite[Proposition 2.3.6]{effros-ruan})
by a family of norms $ \{  \| \cdot \|_n \} _{ n > 0 }$,  such that each $ \| \cdot \|_n $   is a norm on $\cV^{ n \times  n } $ and, for all $ X \in \cV^{ n \times n }$, 
$ Y \in \cV^{ m \times m } $, $ T, S \in \mathbb{C}^{ n \times n } $, we have that:
\begin{itemize}
\item[$\bullet $ ]  $ \| X \oplus Y \|_{ n + m } = \text{max}\{ \| X \|_n , \| Y \|_m \}         
$;
\item[$\bullet$ ] $ \| T X S \|_n \leq \| T \| \| X \|_n \| S \|$, where $ \| \cdot \| $ denotes the usual operator norm of complex matrices.
\end{itemize}

We will consider the operator spaces structures on $ \mathbb{C}^{ m } $
 given by the  $\| \cdot \|_{ \infty } $, $ \| \cdot \col $, and $ \| \cdot \|_{ \text{row} } $,
where, for 
$ X = ( X_1, \dots, X_m ) \in  ( \mathbb{ C}^{ N \times N } )^{ m } \simeq ( \mathbb{ C}^ m )^{ N \times N } $
and $ \| \cdot \| $ the usual operator norm in $ \mathbb{C}^{ N \times N } $
\begin{align*}
&
\| X \|_{ \infty }
 = 
\text{max} \{ \| X_1 \|, \dots,  \| X_m \| \},
\\
&
\| X \col 
 =
\| \sum_{ i = 1}^m X_i^\ast X_i \|^{ \frac{ 1 }{ 2 } },
\\
&
\| X \row  
=
\| \sum_{ i = 1}^m X_i X_i^\ast  \|^{ \frac{ 1 }{ 2 } }.
\end{align*}

 \noindent 
For the norm $ \| \cdot \|_{ \infty } $, the non-commutative unit ball is the non-commutative polydisc
\[
( \mathbb{D}^m )_{ \text{nc} } 
=
\coprod_{ N =1 }^\infty \{ ( X_1, \dots, X_m ) \in ( \mathbb{C}^{ N \times N } )^m : \| X_j \| < 1,\  j = 1, \dots, m \}.
\]

\noindent
For the norms $ \| \cdot \col $, respectively $ \| \cdot \row $, the non-commutative unit balls are given by
\[
(  \mathbb{B}^m )_{ \text{nc} }
=
 \coprod_{ N =1}^\infty
\{ ( X_1, \dots, X_m ) \in ( \mathbb{C}^{ N \times N } )^m : \sum_{ i =1}^m X_i^\ast X_i <  I_N  \},
\]
respectively by
\[
(  \mathbb{B}^m_{ \text{row} }  )_{ \text{nc} }
=
 \coprod_{ N =1}^\infty
\{ ( X_1, \dots, X_m ) \in ( \mathbb{C}^{ N \times N } )^m : \sum_{ i =1}^m X_i  X_i^\ast  <  I_N \}.
\]

Identifying the components from $ (\mathbb{C}^{ N \times N } )^m $
of
$  ( \mathbb{D}^m )_{ \text{nc} }  $, 
 $ (  \mathbb{B}^m )_{ \text{nc} } $,
 respectively 
 $ (  \mathbb{B}^m_{ \text{row} }  )_{ \text{nc} } $
with the corresponding subsets of $ \mathbb{C}^{ m N^2 } , $
the Shilov boundaries for the commutative algebras of complex analytic functions 
in $ m N^2 $ variables, as shown in \cite[Example 1.5.51]{upmeier}, are 
$ \mathcal{U}(N)^m $ in the case of $  ( \mathbb{D}^m )_{ \text{nc} }  $,
respectively the set of all isometries and coisometries of 
$ \mathbb{C}^{ mN \times N } $ for $ (  \mathbb{B}^m )_{ \text{nc} } $,
respectively of 
$ \mathbb{C}^{  N \times mN } $ for $ (  \mathbb{B}^m_{ \text{row} }  )_{ \text{nc} } $. 
Since 
$ \mathbb{C}^{ mN \times N } $
does not have any coisometries, it follows that for the case of 
 $ (  \mathbb{B}^m )_{ \text{nc} } $
the above Shilov boundary is
\begin{align*}
 \dBN &= \{ ( X_1, \dots, X_m ) \in ( \mathbb{C}^{ N \times N } )^m : \sum_{ i = 1 }^m X_i^\ast X_i = I_N \}\\
 &=  \{ ( X_1, \dots, X_m) \in ( \mathbb{C}^{ N \times N } )^m :
  \text{there exists some } U \in \mathcal{U}(mN) \\
 & \hspace{ 4 cm}  \text{  such that  } 
 U  \left[ I_N \ 0 \ \dots  \ 0   \right]^T 
 = \left[ X_1 \ \dots \ X_m  \right]^T
\},
\end{align*}
where, for $ A \in \mathbb{C}^{ n \times m } $, the notation $ A^T $ stand for the matrix transpose of $ A $.

Similarly, since 
$ \mathbb{C}^{ N \times mN } $
does not have any isometries, the Shilov boundary for the case of
$ (  \mathbb{B}^m_{ \text{row} } )_{ \text{nc} } $
is
\[
\left(\partial_S ( \mathbb{B}^m_{ \text{row} })_{\text{nc}}\right)_N
=
\{ ( X_1, \dots, X_m ) \in ( \mathbb{C}^{ N \times N } )^m : \sum_{ i = 1 }^m X_i  X_i^\ast = I_N \}.
\]

To simplify the writing in the next section,  we will denote
 $ \mathcal{U}(N)^m $  by  $ \dDN $.
 The natural  measure on $\dDN $ is the $m$-fold product measure $\mu_N $ of the Haar measure on 
$ \mathcal{U}(N) $. 
Corollary \ref{cor:24} then yields that for any  $ v, w \in \cF_m $ we have:
\begin{align}
& \int_{ \dDN }
 \text{Tr}\left( \left( X^w \right)^\ast X^v \right)  d\mu_N =0
  \ \text{if}\ | v | \neq | w | \label{eq:021} \\
& \lim_{ N \rightarrow \infty } \int_{ \dDN } 
\frac{1}{N}
\text{Tr}\left( \left( X^w \right)^\ast X^v \right)  d\mu_N
 = \delta_{ v, w }. \label{eq:022}
\end{align}

For the case of $ (  \mathbb{B}^m )_{ \text{nc} } $, note that
 the group $ \mathcal{U} ( m N ) $ acts transitively on
 $ \dBN $ via
  $
   \left[  X_1 \dots X_m \right]^T \mapsto 
   U \cdot \left[ X_1 \dots X_m \right]^T
  $. 
  Moreover, denoting
 \[
  H(m, N) = \{ I_n \oplus U : U \in \mathcal{U} ( ( m-1) N ) \},
  \]
 we have that $ H ( m, N ) $ is a compact subgroup of
 $ \mathcal{U} (mN ) $ which is the stabilizer 
 of
  $ \left[ I_N \ 0 \ \dots \ 0 \right]^T \in \dBN $. Hence
 $ \dBN $ is isomorphic to 
 $ \mathcal{U}( mN ) / H(m,N ) $
 and (see \cite[Theorem 2.49]{folland}) there exists a unique
 Radon measure $ \nu_N $ of mass 1 on $ \dBN $ invariant under the action of $ \mathcal{U} ( mN ) $
 and  for any continuous function 
 $ f : \mathcal{U}( mN) \lra \mathbb{C} $
 we have that
 \begin{equation}\label{eq:folland}
 \int_{ \mathcal{U}(mN)} f (U) d\mathcal{U}_{ mN}(U) = 
 \int_{ \dBN  } \int_{  H(m, N) } f(U V )
 d\mathcal{U}_{ ( m -1 ) N }( V ) d\nu_N( U H(m,N ) ).
 \end{equation}

For $ 1 \leq i \leq mN $ and $ 1 \leq j \leq N $  and 
$ u_{ i, j } : \mathcal{U}( mN ) \rightarrow \mathbb{C} $ 
as defined in Section \ref{haar}, a simple verification gives that for all 
$ U \in \mathcal{U}(mN)$ and $ V \in H( m, N ) $
\begin{equation}\label{eq:023}
 u_{ i , j } ( U ) = u_{ i,  j } ( U \cdot V ).
\end{equation}
Fix now  $ f \in \text{Alg}\{ u_{ i, j }, \overline{ u_{ i, j } } : 1 \leq i \leq mN, 1 \leq j \leq N \}$.
 For all $ U\in\mathcal{U}(mn) $, equation (\ref{eq:023}) implies 
\begin{equation}\label{eq:024}
 \int_{  H(m, N) } f(U V )
d\mathcal{U}_{ ( m -1 ) N }( V ) =
 \int_{ H( m, N ) } f ( U ) d\mathcal{U}_{ ( m -1 ) N }( V )  = f( U ).
\end{equation}
Define $ \widehat{f} $ on $ \dBN $ via $ \widehat{f} ( U H ( m, N ) ) = f( U) $. 
 From (\ref{eq:023}), $ \widehat{f} $ is well-defined. Moreover, equations \eqref{eq:folland} and \eqref{eq:024} gives
\begin{equation}\label{eq:025}
\int_{ \dBN } \widehat{f} ( U H(m, N ) ) d\nu_N (U H( m, N ) ) = \int_{ \mathcal{U}(mN ) } f ( U ) d\mathcal{U}_{mN} ( U ).
\end{equation}
Hence, Corollary \ref{cor:2} (more precisely, its analogue for the first $mN \times N$ matricial column)
implies that for all $ v, w \in \cF_m $ we have:
\begin{align}
&\int_{ \dBN } \text{Tr}\left( \left( X^w \right)^\ast X^v \right) 
d\nu_N(X) = 0\ \text{if} \  | v | \neq | w |\label{eq:031}\\
&\lim_{ N \rightarrow \infty } \int_{ \dBN } \frac{1}{N}\text{Tr} \left( 
 \left( X^w \right)^\ast X^v \right) d\nu_N(X)
= \delta_{ v, w } \frac{1}{m^{ | v |  } }.\label{eq:032}
\end{align}

For the case of $(\mathbb{B}^m_{ \text{row}})_{\text{nc}}$,  a similar argument as above gives that there exists $ \nu^\prime_N $, 
a unique Radon measure 
of mass 1   on 
$\left(\partial_S( \mathbb{B}^m_{ \text{row} })_{\text{nc}}\right)_N$ invariant under the action
of $ \mathcal{U}(mN) $, 
and that the pair $ ( \left(\partial_S( \mathbb{B}^m_{ \text{row} })_{\text{nc}}\right)_N, \nu^\prime_N ) $ 
also satisfies the equalities (\ref{eq:031}) and (\ref{eq:032}).


\section{Main results}

The present section will  address properties of certain $ H^2 $ Hardy spaces 
associated to the non-commutative unit balls for the operator norms
 $ \| \cdot \|_\infty $ and $ \| \cdot \col $ on $ \mathbb{C}^m $. 
 Since both 
 $ ( \dBn, \nu_n ) $ 
 and 
 $  (( \partial_S( \mathbb{B}^m_{ \text{row} })_{\text{nc}})_n, \nu^\prime_n ) $
 satisfy (\ref{eq:031}) and (\ref{eq:032}), similar results
 to the case of $ \| \cdot \col $ can be stated for the setting of
  $ \| \cdot \row $.
  
  For $ \Omega $ either $ (\mathbb{B}^m)_{\text{nc}} $ or $ (\mathbb{D}^m)_{\text{nc}} $, consider the algebras
  \begin{align*}
  \cA_{ \Omega } =&\{  f:  \Omega \lra \mathbb{C}_{ \text{nc } } : \ 
   f  \text{ is a non-commutative function, locally bounded on slices}\\  
  & \hspace{6 cm}  \text{   separately in every matrix dimension} \}.
  \end{align*}
 Equation  (\ref{eq:nc1}) gives that for all $ f \in \cA_\Omega $
 there exists a family of complex numbers
  $ \{ f_w \}_{ w\in\cF_m} $
   with $ f_{\emptyset} = f( 0 ) $, such that for all $ X \in \Omega $
  \begin{equation}\label{eq:nc2}
  f(X) = \sum_{ l = 0 }^\infty\big( \sum_{ w\in \cFml } X^w f_w \big),
  \end{equation}
where, for any positive integer $N$, the series converges absolutely and uniformly on compacta of $\Omega_N = \Omega \cap \mathbb{C}^{ N \times N }$.

For $ f \in \cA_\Omega $ as above and $ X \in (\mathbb{C}^m)_{ \text{nc} } $, we will denote
 $ \displaystyle  f^{ [ l ] } ( X ) = \sum_{ w \in \cFml  }  X^w f_w $. Remark that 
$ f^{ [ l ]} (rX) = r^l \cdot f^{ [ l ] } ( X ) $ for any real $ r $; also, if $ l \neq p $, then 
$$ \displaystyle
\int_{ \partial(\Omega, N )  }
 \text{Tr}\left( ( f^{ [ l ] } (X) )^\ast f^{ [ p ] } (X ) \right) d\omega_N = 0.
$$

\begin{thm}\label{thm:1.a} \emph{(i)} If $ f \in \cA_{ \mathbb{D}^m } $ and 
$ r \in (0, 1) $, then: 
\begin{align*}
f_w  & =
\lim_{ N \lra \infty } \frac{ 1 } { r^{ | w | } }
\int_{\dDN} \frac{1 } { N } \textrm{Tr}  \left( ( X^w )^\ast f( r X )  \right) d \mu_N\\
& =
\lim_{ r \lra 1^- } \lim_{ N \lra \infty } 
\int_{ \dDN  } \frac{ 1 }{ N } \text{Tr} \left( ( X^w )^\ast f( r X )  \right) d \mu_N.
\end{align*}
\emph{(ii)} If $ f \in \cA_{ \mathbb{B}^m } $ and $ r \in (0, 1 ) $, then
\begin{align*}
f_w  & =
\lim_{ N \lra \infty } \frac{ 1 } { r^{ | w | } }\cdot {m^{ | w |  } }
\int_{ \dBN} \frac{1 } { N } \textrm{Tr}  \left( ( X^w )^\ast f( r X )  \right) d \nu_N\\
& =
\lim_{ r \lra 1^- } \lim_{ N \lra \infty } 
 {m^{ | w | } }
\int_{ \dBN  } \frac{ 1 }{ N } \text{Tr} \left( ( X^w )^\ast f( r X )  \right) d \nu_N.
\end{align*}
\end{thm}

\begin{proof} For any positive integer $ N $, relations (\ref{eq:021}) and (\ref{eq:nc2}) give
(notice that we can interchange the integral and the infinite sum since the convergence
is uniform on $r \dDN$)  
\begin{align*}
\int_{ \dDN } \text{Tr} \left( ( X^w )^\ast f (r X )  \right) d\mu_N
& =
 \sum_{ l = 0 }^\infty \left( \sum_{ v\in \cFml } \int_{ \dDN }
 \text{Tr}\left( (  X^w )^\ast X^v \right) \cdot r^{ | v | } f_{ v } d\mu_N \right)\\
& = 
\sum_{ \substack{ v \in \cF_m \\ | v | = | w | } }
r^{ | v | } f_{ v }\cdot
 \int_{ \dDN }
\text{Tr}\left( (  X^w )^\ast X^v \right)   d\mu_N,
\end{align*}
and the equalities from {(i)} follow from equation (\ref{eq:022}). 

The argument for {(ii)} is similar, using equations (\ref{eq:031}), (\ref{eq:nc2}) and (\ref{eq:032}).
\end{proof}

\begin{defn}\label{defn:omega1}
 For $ (\Omega, d\omega_N) $ 
either $ ((\mathbb{B}^m)_{\text{nc}}, d\nu_N)$ or $ ( (\mathbb{D}^m)_{\text{nc}}, d\mu_N)  $, we define
\[
 \displaystyle H^2( \Omega) =  
\{ f \in \cA_\Omega :  S( f ) =
 \sup_N  \sup_{ r < 1 }
 \int_{  (\partial_S \Omega)_N }
 \frac{ 1} { N } \text{Tr} \left(  f ( r X )^\ast f ( r X )  \right) d \omega_N < \infty
\}. 
\] 
\qed
\end{defn}

\begin{thm}\label{thm:3.3}
\emph{(i)}  If $ f \in H^2 (  ( \mathbb{D}^m )_{ \text{nc} } ) ,  $
 then 
 $\displaystyle  \sum_{ l = 0 }^\infty ( \sum_{  w \in \cFml  } | f_w |^2 ) < \infty $.

 \emph{(ii)} If $ f \in H^2 (  ( \mathbb{B}^m )_{ \text{nc} } ) ,$  then 
 $\displaystyle  \sum_{ l = 0 }^\infty ( \frac{1}{m^{ l  } } 
\sum_{ w \in \cFml } | f_w |^2 ) < \infty $.
\end{thm}

\begin{proof}
 For  (i), note first that if $ r \in ( 0, 1 ) $, 
  $ f \in H^2 ( ( \mathbb{D}^m ) _{ \text{nc} } )$
and $ X \in \dDN $,  then, as in equation (\ref{eq:nc2}), we have that 
  $ \displaystyle f ( r X ) = \sum_{ l =0 }^\infty r^l f^{ [ l ] } (X) $  and the series is absolutely convergent.

 Therefore equation (\ref{eq:021})  implies that 
\begin{align*}
\int_{ \dDN } \frac{1}{N}& \text{Tr} ( f (r X)^\ast f( r X ) ) d \mu_N =
 \sum_{ l = 0 }^\infty  r^{ 2 l }
 \int_{ \dDN } \frac{1}{N} \text{Tr} ( 
f^{[l]}(X)^\ast f^{[l]}( X ) ) d\mu_N \\
&\hspace{1cm}= \sum_{ l = 0 }^\infty  
r^{ 2 l } [ a_l + b_{ l, N } ]
\end{align*}
where  
$ \displaystyle a_l = \sum_{ | w | = l } | f_w |^2 ,$ and 
\[
 \displaystyle b_{ l, N } = 
\sum_{  | v | = | w | = l  } \overline { f_ w } f_v
 \cdot \int_{ \dDN } \frac{1}{N} \text{Tr} (  ( X^w)^\ast X^v ) d\mu_N  - a_l.
\]

 Since 
 $ f \in H^2( ( \mathbb{D}^m )_{ \text{nc} })$, according to Definition \ref{defn:omega1}, we have that
  \[
  \sup_{ N } \sup_{ r < 1 } \sum_{ l = 0 }^\infty r^{ 2 l } \cdot ( a_l + b_{ l, N } ) = S(f) < \infty.
  \]

For each $ X $ and each $ l, N $, the matrix
$ f^{[ l ] }(X)^\ast f^{ [ l ] }(X) $ is positive, 
henceforth $ a_l + b_{l, N } \geq 0 $ thus, for each positive integer $ L $ and each $ r \in (0, 1) $,
\[
\sum_{ l =0}^L r^{2l} (a_l + b_{l, N}) \leq S(f).
\]

 \noindent But relation (\ref{eq:022}) implies that
  for each $ l $, 
 $ \displaystyle \lim_{ N \rightarrow \infty } b_{ l, N }=0 ,$  
  thus 
  $ \displaystyle \sum_{ l =0}^L r^{2l}a_l \leq S(f) $ for each $ L $ and each $ r $, henceforth
 $\displaystyle \sup_{ r < 1 } \sum_{ l = 0 }^\infty r^{ 2 l } a_l < \infty $, and, 
since $ a_l \geq 0$, we obtain $ \displaystyle \sum_{ l = 0 }^\infty a_l < \infty $.

The argument for part (ii) is analogous, utilizing equations (\ref{eq:031})  and (\ref{eq:032}).
\end{proof}

\begin{prop}\label{prop:34}
 For $ f, g \in H^2( \Omega) $ and $ r > 0 $, define
 \[
  \varphi_{ f, g }(r) =
   \lim_{ N \lra \infty }
    \int_{  (\partial_S \Omega)_N }
    \frac{1}{N} \Tr \left( g( rX)^\ast f( rX ) \right)
d\omega_N.
  \]
  
   Then  $ \varphi_{ f, g } $ exists for $ r $ small and equals 
   $  \displaystyle \sum_{ l=0}^\infty  r^{ 2l} \big(  \sum_{ w \in \cFml } \overline{ g_w} {f_w} \big) $ for $ \Omega =   ( \mathbb{D}^m )_{ \text{nc} }  $, respectively equals
   $  \displaystyle \sum_{ l=0}^\infty  r^{ 2l} \big(  \sum_{ w \in \cFml } \overline{ g_w} {f_w} \big) $ for $ \Omega =   ( \mathbb{B}^m )_{ \text{nc} }  $;
     in both cases the latest extends analytically  on $ (0, 1) $ and continuously on $ [ 0, 1 ] $ to a function that we will denote by $ \widetilde{ \varphi_{ f, g}}$.
\end{prop}
\begin{proof}

 Since $ f, g \in H^2( \Omega) $, equations (\ref{eq:021}) and (\ref{eq:031}) give that
\begin{equation}\label{eq:prop34}
\int_{  (\partial_S \Omega)_N }
    \frac{1}{N} \Tr \left( g( rX)^\ast f( rX ) \right)
d\omega_N
=
\int_{  (\partial_S \Omega)_N } \sum_{ l=0 }^\infty
\frac{1}{N} \Tr \left( g^{[l]}(rX)^\ast f^{[l]}( rX ) \right)
d\omega_N.
\end{equation}

On the other hand
\begin{align*}
\frac{1}{N} \Tr \left( g^{[l]}(rX)^\ast  \right.&  \left.f^{[l]}( rX )\right) 
 \leq   \| g^{ [ l]} ( rX)\| \cdot \| f^{ [ l ] } (r X ) \| \\
\leq & ( r^l \sum_{ v \in \cFml } | g_v | \cdot \| X^v \| ) \cdot (r^l \sum_{ w \in \cFml } | f_w | \cdot \| X^w \| ) \\
\leq & (m\cdot r)^{2l} ( \frac{1}{m^l} \sum_{ v \in cFml } | g_v | ) \cdot ( \frac{1}{m^l } \sum_{ w \in \cFml } | f_w | )  \\
\leq & (m\cdot r)^{2l}  \sup_{ v \in \cFml} | g_w | \cdot \sup_{ w \in \cFml} | f_w |
 \leq  (m\cdot r)^{2l} ( \sup_{ v \in \cFml} | g_w |^2
 + \sup_{ w \in \cFml} | f_w |^2 ) .
\end{align*}
 Theorem \ref{thm:3.3} implies then that if $ \Omega =   ( \mathbb{D}^m )_{ \text{nc} }  $ and $ r < \frac{1}{m} $, respectively if $ \Omega = (  \mathbb{B}^m )_{ \text{nc} } $ and 
 $ r < \frac{1}{m^2} $ the sum from the right-hand side of equation (\ref{eq:prop34}) is absolutely convergent, uniform in $ N $ and $ X $, hence 
\[
\int_{  (\partial_S \Omega)_N }
    \frac{1}{N} \Tr \left( g( rX)^\ast f( rX ) \right)
d\omega_N
=
\sum_{ l=0 }^\infty
\int_{  (\partial_S \Omega)_N } 
\frac{1}{N} \Tr \left( g^{[l]}(rX)^\ast f^{[l]}( rX ) \right)
d\omega_N,
\]
 with the right-hand side absolutely convergent uniform in $ N$.
  
   From equations (\ref{eq:022}) and (\ref{eq:032}), 
  \[
   \displaystyle \lim_{ N \rightarrow \infty }\int_{ \partial(\Omega, N ) } \frac{1}{N}\text{Tr} 
  \left(g^{ [ l ] } (  r X)^\ast f^{ [ l ] } ( r X )
  \right) d\omega_N 
  =
  \left\{
  \begin{array}{lc}
   \displaystyle r^{ 2 l}\sum_{  w \in \cFml  }\overline{ g_w}f_w &  
  \text{if  } f, g \in H^2( ( \mathbb{D}^m)_{\text{nc} } ) \\
   \displaystyle r^{2l}\sum_{  w \in \cFml } \frac{1}{ m^{ l  } } \overline{ g_w}f_w &  
  \text{if  } f, g \in H^2( ( \mathbb{B}^m)_{\text{nc} } ) \\
  \end{array}
  \right.
   \]
  and the conclusion follows from Theorem \ref{thm:3.3}.
  
\end{proof}

 An immediate consequence of the Proposition above is the following Theorem.

\begin{thm}\label{thm:innerprod}
${}$\newline
\emph{(i)} With the notations from Definition \ref{defn:omega1}, $ H^2(\Omega_{\text{nc} } ) $
are inner-product spaces, with the inner product given by
\[
 \langle f , g \rangle = 
 \lim_{ r \lra 1^- } \widetilde{\varphi_{f, g}}(r) =
\left\{
\begin{array}{lc}
 \displaystyle \sum_{ { w \in \cF_m } }\overline{ g_w}f_w &  
\text{if  } f, g \in H^2( ( \mathbb{D}^m)_{\text{nc} } ) \\
 \displaystyle \sum_{ { w \in \cF_m  } } \frac{1}{ m^{ l  } } \overline{ g_w}f_w &  
\text{if  } f, g \in H^2( ( \mathbb{B}^m)_{\text{nc} } ) \\
\end{array}
\right.
\]
\emph{(ii)} $\{ X^{ w } \}_{ w \in \cF_m } $  is a complete orthonormal system in
 $ H^2 ( (\mathbb{D}^m)_{\text{nc}} ) $ 
and, for all \\
 $ f \in H^2 ( (\mathbb{D}^m)_{\text{nc}} )$, 
we have that 
 $ f_w = \langle f, X^w \rangle $ and $ \displaystyle  f =\sum_{ w \in \cF_m } f_w X^w $
 in $ H^2 ( (\mathbb{D}^m)_{\text{nc}} ) $.

\emph{(iii)} $\{  m^{  \frac{ | w |  } { 2 } } X^w \}_{ w \in \cF_m } $
  is a complete orthonormal system
in $ H^2 ( ( \mathbb{B}^m )_{ \text{nc} } ) $ and, for all
 $ f \in H^2 ( ( \mathbb{B}^m )_{ \text{nc} } ) $,
  we have that
$ f_w = \langle f,  m^{ | w | } X^w \rangle $ and $ \displaystyle  f =\sum_{ w \in \cF_m } f_w X^w$
in $ H^2 ( ( \mathbb{B}^m )_{ \text{nc} } ) $.
\end{thm}

\begin{proof}
Part (i) follows from Proposition \ref{prop:34} and since the sums are convergent from Theorem \ref{thm:3.3}.

 The parts (ii) and (iii) are simple consequences of Theorem
  \ref{thm:1.a}  and equations (\ref{eq:022}) and (\ref{eq:032}).
\end{proof}

\begin{remark} The limit over $ N $ in Theorem \ref{thm:innerprod}(i) is not the supremum. For example, if $ m = 2 $ and
$ f(X) = X_1X_2 + X_2X_1 $, then 
\[
\int_{ \dDN} \frac{1}{N} \text{Tr} \left( f(rX)^\ast f(rX) \right) d\mu_N = 2r^2(1+\frac{1}{N^2})
\]
\end{remark}

\begin{defn}
As before, $ \Omega $ will denote either $ \mathbb{D}^m $ or $ \mathbb{B}^m $.

For $ X \in \Omega_{ \text{nc} } $, define the map 
$ E^X_ { \Omega } : H^2( \Omega_{ \text{nc} } ) \lra \mathbb{C}^{  N \times N } $
via
$
 E^X_\Omega ( f ) = f( X ) . $ 

\noindent Let 
$\mathcal{B}_{ \Omega, N } = \{ X \in  \Omega_{ \text{nc} } \cap \mathbb{C}^{ N \times N } : \ 
E^X_{\Omega} \text {is a bounded map} \}
$
and
$\displaystyle  \mathcal{B}_\Omega = \coprod_{ N =1}^\infty \mathcal{B}_{ \Omega, N } .$

\end{defn}

For $ p > 0 $, we define the Hilbert space
\[
l^2_p ( \cF_m ) = 
\{
\{ f_w \}_{ w \in \cF_m } : \ f_w \in \mathbb{C} , 
\| \{ f_w \}_{ w \in \cF_m } \|_{ 2, p } = \sum_{ l=0}^\infty  
( \sum_{ w \in \cFml } \frac{1}{p^l }\overline{f_w }f_w ) 
< \infty 
\}.
\]

\begin{prop}\label{prop:1}
With the above notations, we have that
 \begin{enumerate}
\item[(i)]$ \displaystyle \mathcal{B}_{ \mathbb{D}^m } = \{ X \in ( \mathbb{D}^m)_{ \text{nc} } :
 \text{ the series}  
\sum_{ l = 0 }^\infty ( \sum_{ w \in \cFml } f_w X^w )
 \text{ converges}
\text{ for any} \\ \text{sequence }
 \{ f_w \}_{ w \in \cF_m} \in l^2( \cF_m) 
   \}
$
\item[(ii)]
$ \displaystyle \mathcal{B}_{ \mathbb{B}^m } = \{ X \in ( \mathbb{B}^m)_{ \text{nc} } :
 \text{ the series} 
 \sum_{ l = 0 }^\infty( \sum_{ w \in \cFml } f_w X^w ) 
\text{ converges}
\text{ for } \\ \text{ any sequence }
 \{ f_w \}_{ w \in \cF_m} \in l_m^2( \cF_m) 
   \}
$
\end{enumerate}
\end{prop}

\begin{proof}
 Suppose that  $ X \in \mathbb{D}^m \cap \mathbb{C}^{ N \times N }$  is such that 
$ \displaystyle \sum_{ l = 0}^\infty ( \sum_{  w \in \cFml } f_w X^w )  $
 converges for all $ \{ f_ w \} \in l^2 ( \cF_m ) $
and consider the linear map
$ \widetilde{\ED^X} : l^2( \cF_m ) \lra \mathbb{C}^{ N \times N } ,$
 given by
\[
\widetilde{\ED^X} ( \{ f_w \}_{ w \in \cF_m } ) = 
\sum_{ l = 0}^\infty ( \sum_{   w \in \cFml } f_w X^w ). 
\]

For every $  l $, define also 
\[
 \ED^{X, l} ( \{ f_w\} ) = \sum_{s=0 }^l (  \sum_{  w \in \cFml  }  f_w X^w  ).
\]
From the initial assertion, $ \widetilde{\ED^X} $ is the pointwise limit of  $ \{\ED^{X, l}\} _{l>0 }. $ 
 Each  $ \ED^{X, l} $ is a bounded linear operator from $ l^2( \cF_m ) $ to $ \mathbb{C}^{ N \times N } $, 
so Banach-Steinhaus Theorem gives that $ \widetilde{\ED^X} $ is bounded.

Take now $ f \in H^2 ( (\mathbb{D}^m)_{ \text{nc} } ) $. From Theorem \ref{thm:3.3},
 the sequence
 $\{ f_w\}_{ w\in \cF_m} $ of its Taylor-Taylor coefficients is in
 $ l^2( \cF_m ) $ and its norm
 coincides to the norm of $ f $ in  $ H^2 ( (\mathbb{D}^m)_{ \text{nc} } ) ,$ 
 hence the operator $ \ED^X $ is bounded and
 $ \| \ED^X \| \leq \| \widetilde{\ED^X} \| $.

 For the  converse,  fix
 $ \{ f_w \}_w \in l^2 (\cF_m ) $
 and, for all $ l > 0 $ consider the functions 
$ \alpha_l : ( ( \mathbb{D}^m)_{\text{nc} })_N \longrightarrow \mathbb{C}^{ N \times N } $ 
given by 
$ \displaystyle \alpha_l (X) = \sum_{ | w | \leq l } f_w X^w .$
 The sums are finite, therefore 
$ \alpha_l \in H^2( (\mathbb{D}^m)_{ \text{nc} } ) ,$
 hence, if
 $ X \in \mathbb{D}^m \cap \mathbb{C}^{ N \times N } $, then
\begin{align*}
\| \alpha_{ l + s }(X) - \alpha_{ l  }(X) \| &
 \leq \| \ED^X \| \cdot \| \alpha_{ l+ s} - \alpha_{ l } \|_{ H^2 ( ( \mathbb{D}^m)_{\text{nc} } ) }\\
& \leq \| \ED^X \| \cdot ( \sum_{ l < | w | \leq l+ s } | f_w|^2 ).
\end{align*}
Since the sequence 
$\{ \sum_{  | w | \leq l } | f_w|^2 \}_{ l \geq 0 } $ is Cauchy, it follows that $ \{ \alpha_l \}_l $
 is also a Cauchy sequence, 
therefore the series
 $ \displaystyle  \sum_{ l = 0 }^\infty ( \sum_{   w \in \cFml  } f_w X^w ) $ converges.

The argument for part (ii) is similar, replacing $ l^2(\cF_m ) $ to $ l^2_m ( \cF_m ) $ and using second parts of Theorem \ref{thm:3.3} and of  relation  (\ref{eq:coeff}).

\end{proof}

\begin{thm}\label{thm:3}
 For $ p > 0 $ , define 
\[ \displaystyle  \Upsilon^m_p = \{  X \in ( \mathbb{C}^m )_{ \text{nc} } : 
   \sum_{  w \in \cF_m  }   p^{ | w | } (X^w)^\ast X^w   
\text{converges} \}. \]
 Then 
$\mathcal{B}_{ \mathbb{D}^m } = ( \mathbb{D}^m )_\text{nc} \cap \Upsilon_1^m $
and 
$\mathcal{B}_{ \mathbb{B}^m } = ( \mathbb{B}^m )_\text{nc} \cap \Upsilon_m^m $.

Moreover, if $ X \in \Upsilon^m_p \cap \mathbb{C}^{N \times N} $, 
then 
 $ \{ X^w \}_{ w \in \cF_m } \in l^2_{\frac{1}{p }} ( \cF_m ) \otimes \mathbb{C}^{ N \times N } . $ 
\end{thm}

\begin{proof}

 Suppose that $ X \in \mathcal{B}_{ \mathbb{D}^m, N }$.
 Since, according to  Proposition \ref{prop:1}(i),  
the series \\
$\displaystyle  \sum_{ l =0 }^\infty (  \sum_{ w \in \cFml  } f_w X^w ) $
 converges for any 
$ \{ f_w \}_{ w \in \cF_m }$ from  $  l^2 ( \cF_m ) ,$
 it follows that the series 
 $ \displaystyle \sum_{ l = 0 }^\infty ( \sum_{ w \in \cFml  } f_w e^\ast X^w \widetilde{e} )  $
also converges for any $ e, \widetilde{e} \in \mathbb{C}^N $. 
The Riesz Representation Theorem  gives that
 $ \{ e ^\ast X^w \widetilde{e} \}_{ w \in \cF_m } \in l^2 ( \cF_m ) $, therefore
also  the series 
$
\displaystyle
\sum_{ l =0}^\infty \sum_{ w \in \cFml } e^\ast (X^w)^\ast \widetilde{e}  X^w 
$
converges for all $ e, \widetilde{e} \in \mathbb{C}^N .$

Taking $ e, \widetilde{e} $ from the cannonical basis of $ \mathbb{C}^N $, we get that
 $  \sum_{ w \in \cF_m } (X^w)^\ast X^w $ converges on each entry, therefore in $ \mathbb{C}^{ N \times N } $.

The argument for $( \mathbb{B}^m)_\text{nc} $ is similar.

 Suppose now than $ X \in \Upsilon^m_p \cap \mathbb{C}^{N \times N} $.
Then
$  \displaystyle \sum_{ w \in \cF_m }p^{ | w | } (X^w)^\ast X^w  $  
also converges entrywise, and, since the $ (j, j) $-entry of the series equals
\[
\sum_{ w \in \cF_m }p^{ | w | }(  \sum_{ l=1}^N \overline{ x^{ ( w ) }_{ l, j } } x^{ ( w ) }_{ l, j } ) 
= 
\sum_{ l =1 }^N  (  \sum_{ w \in \cF_m } p^{ | w | }\overline{ x^{ ( w ) }_{ l, j } } x^{ ( w ) }_{ l, j } ) 
\]
where $ x^{ ( w )}_{ l, j} $ is the $ (l, j) $-entry of $ X^w  $,
 it follows that $ \{ x_{l, j}^{ ( w ) } \}_{ w \in \cF_m } \in l^2_{ \frac{1}{p} }( \cF_m ) $ for all $ l, j $. 
In particular $ \{ X^w \}_{ w \in \cF_m } \in l^2_{ \frac{1}{p}}( \cF_m ) \otimes \mathbb{C}^{ N \times N } $.

\end{proof}

\begin{remark}\label{remark:3.10}
\emph{(i)} $ \Upsilon_1^m \not\subset (\mathbb{D}^m)_\text{nc}$ 
and
 $ \Upsilon_m^m \not\subset (\mathbb{B}^m)_\text{nc}$.

\emph{(ii)} 
If $ X = (X_1, X_2, \dots, X_m ) \in ( \mathbb{C}^{ N \times N } )^m $ is such that 

\[ X^\ast X_1 + X_2^\ast X_2 + \dots + X_m^\ast X_m < \frac{1}{p} \]
then $ X \in  \Upsilon_p^m $.
 In particular, 
 $ \frac{1}{\sqrt{m}}  ( \mathbb{D}^m )_\text{nc} \subset \mathcal{B}_{\mathbb{D}^m}$
and
$ \frac{1}{\sqrt{m}}  ( \mathbb{B}^m )_\text{nc} \subset \mathcal{B}_{\mathbb{B}^m}$.

\end{remark}

\begin{proof}

 For part (i), it suffices to take $ X = (X_1, 0, \dots, 0) $ with $ X_1 $ nilpotent with norm larger than 1. Then 
 $ X \in \Upsilon_p^m $ for any $ p >0 $, but 
$ X \not\in (\mathbb{B}^m)_\text{nc}, (\mathbb{D}^m)_\text{nc} $.

For part (ii),  suppose that 
 $ X^\ast X_1 + X_2^\ast X_2 + \dots + X_m^\ast X_m < 
\frac{ \theta}{p} $ 
for some  $ 0< \theta < 1$.
Denote by $ \displaystyle X^{ [ l ] } =p^l \sum_{  w \in \cFml } (X^w )^\ast X^w .   $ 
Then
\[
0\leq X^{ [ l + 1 ] } = \sum_{  w \in \cFml  } p^l
 (X^w )^\ast \left(  p \sum_{ k =1}^m X_k^\ast X_k  \right) X^w   < \theta  X^{ [ l ] } \\
\]
hence $ \displaystyle \sum_{ w \in \cF_m } (X^w)^\ast X^w < \frac{ 1 }{1 - \theta  }. $

\end{proof}

\begin{defn}
For $ p > 0 $, we will consider the sets
\[
\cK_{ p}  = \{ ( X, Y ) \in ( \mathbb{C}^m)_{ \text{nc} } \times ( \mathbb{C}^m)_{ \text{nc} }  :\
\sum_{ l = 0}^\infty [
\sum_{  w \in \cFml   }p^{ l } X^w \otimes (Y^w )^\ast  ]  \ \text{converges} \}
\]
and the maps 
$ K_{ p}: \cK_{ p } \lra \mathbb{C}_{ \text{nc} }  $,
given by
\begin{align*}
K_{ p } ( X,  Y ) = & \sum_{ l = 0}^\infty [
\sum_{  w \in \cFml  }p^{ l } X^w \otimes (Y^w )^\ast  ] .
\end{align*}
\end{defn}

Theorem \ref{thm:3} implies the following result:
\begin{remark}
 $ ( \Upsilon^m_p  \times \Upsilon^m_p ) \subset \cK_{ p} $.  \hfill $\square$
\end{remark}

Also, note that from the second part of Theorem \ref{thm:3}, 
any sequence $ f = \{ f_w \}_{ w \in \cF_m } $ from $ l^2_p (\cF_m) $
  can be identified with a nc-function on
$\Upsilon_p^m $  via 
\[
f( X ) = \sum_{ l=o}^\infty ( \sum_{ w \in \cFml  } f_w X^w ) .
\]

\noindent  
Next, we will consider  the following spaces of nc-functions:
\begin{defn}
For $ p > 0 $ define $ \overline{ H^2_{ m, p } }  $
as follows:
\begin{align*}
\overline{ H^2_{ m, p } }  & = \{ 
f:\Upsilon^m_{ p } \lra \mathbb{C}_{ \text{nc} } : \ f \ \text{is nc-function such that
there exists some sequence}\\
& \{ f_w \}_{ w \in \cF_m } \in l^2_{ p }( \cF_m ) \
\text{such that} \ 
 \displaystyle f(X) = \sum_{ l=o}^\infty \sum_{ w \in \cFml } f_w X^w \}. 
\end{align*}
\end{defn}
Note that $ \overline{ H^2_{ p } }  $ 
are Hilbert spaces with the inner-products inherited from $ l^2_{ p }(\cF_m ) $.
Proposition \ref{prop:11} below shows that in fact they are reproducing kernel Hilbert spaces 
with respect to $ K_{ p } $.

\begin{prop}\label{prop:11}
 Fix $ Y \in \Upsilon_{p}^m \cap \mathbb{C}^{ M \times M } $. With the 
notations above, we have that:
\begin{itemize}
\item[(i)] The map 
$ K_p ( \cdot, Y) :\Upsilon_{p}^m \lra ( \mathbb{C}^{M \times M})_{ \text{nc} }$
is a non-commutative function that belongs to
$ \overline{ H^2_{ p, m } } \otimes \mathbb{C}^{ M \times M } $.
\item[(ii)] for any $ e_1,  e_2  \in \mathbb{C}^M $ and any 
$ f \in \overline{ H^2_{ p , m } } $ , 
\[
\langle f, e_1^\ast K_{ p}( \cdot, Y )  e_2 \rangle_{ l^2_p( \cF_m ) }
= e_2^\ast f ( Y ) e_1.
\]
\end{itemize}

\end{prop}

\begin{proof}

For part (i), first note that for any  $ w \in \cF_m $, the map 
$ X\mapsto p^{ |w | }X^w \otimes (Y^w)^\ast $ 
is a noncommutative function from 
$ \mathbb{C}_{ \text{nc} } $ to $ ( \mathbb{C}^{ M \times M } )_{ \text{nc} } $, 
hence it suffices to prove the convergence in $ \overline{ H^2_{ p, m } } $.

Let $ y^{(w)}_{ i, j } $ be the $ (i, j) $-entry of $ Y^w $.

From Theorem \ref{thm:3}, the sequences $ \{ y^{ ( w )}_{ i, j } \}_{ w \in \cF_m } $ 
are in $ l^2_p( \cF_m ) $ 
for all $ i, j $, hence the map
$ X \mapsto \sum_{ w \in \cF_m } y^{ ( w ) }_{ i, j } X^w $
 is a $ \mathbb{C}_{ \text{nc} } $ -valued non-commutative function from 
$ \overline{ H_{ p, m }^2 } $.

For part (ii), let $ \{ f_w \}_{ w \in \cF_m } \in l^2 ( \cF_m ) $ such that
 $ \displaystyle f ( X ) = \sum_{ w \in \cF_m } f_w X^w $. Then
\begin{align*}
\langle f,  e_ 1 ^\ast K_{ p} ( \cdot, Y )  e_2 \rangle_{ l^2_p( \cF_m ) }
& =
\langle \{ f_w \}_{w\in \cF_m } ,
\{ e_1^\ast ( Y^w)^\ast  e_2 \}_{w \in \cF_m } \rangle_{ l^2_p( \cF_m) }\\
=&
\sum_{ w \in \cF_m } e_2^\ast f_w Y^w e_1 = e_2^\ast f( Y ) e_1.
\end{align*}
\end{proof}

\begin{prop}
 Suppose that $ f $ is a  non-commutative function locally bounded on slices separately in every matrix dimension around $ 0 $ and 
\begin{align*}
 \Phi ( r )
 &= 
\lim_{ N \lra \infty } \int_{ \dDN } \frac{ 1 } { N } \text{Tr} 
( f ( r X )^\ast f ( r X ) ) d \mu_N,\\
\Psi( r ) 
&=
\lim_{ N \lra \infty } \int_{ \dBN } \frac{ 1 } { N } \text{Tr} 
( f ( r X )^\ast f ( r X ) ) d \nu_N.
\end{align*}
 Then  $ f $ extends to a function in $ \overline { H ^2_{ 1, m } } $, 
respectively in  $ \overline { H ^2_{ m, m } } $, 
 if and only if $ \Phi ( r ) $, 
respectively $ \Psi ( r ) $,  exists for all small $ r $ 
$ ( $in which case $ \Phi $, respectively $ \Psi $,  are also analytic at $ 0 $ $)$ and it extends analytically to $ ( 0, 1 ) $ and continously to $ [ 0, 1 ]$ .

Moreover, 
$ \displaystyle \lim_{ r \lra 1^- } \Phi (r ) \geq \| f \|_{ \overline{ H^2_{1, m } } } $, 
respectively 
$ \displaystyle \lim_{ r \lra 1^- } \Psi (r ) \geq \| f \|_{ \overline{ H^2_{m, m } } } $.
\end{prop}

\begin{proof}

Suppose first that $ f $ extends to $ \widetilde{f} \in \overline{H^2_{ 1, m } } $, that is there exists
$ \{ f_w \}_{ w \in \cF_m } $ such that 
\begin{equation}\label{eq:l2}
f(X) = \sum_{ l = 0 }^\infty ( \sum_{ w \in \cFml } f_w X^w )
\end{equation}
for all $ X \in \Upsilon^m_1 $; in particular, Remark \ref{remark:3.10}(ii) gives that the expansion (\ref{eq:l2})
  holds for all
$ X \in \frac{1}{\sqrt{m} } \mathbb{D}^m .$
 
As before, consider the non-commutative functions 
 $ f^{[ l ] } : ( \mathbb{C}^m )_{ \text{nc} } \lra \mathbb{C}_{ \text{nc} } $
given by
$ f^{ [ l ] } (X) = \sum_{ w \in \cFml } f_w X^w  $. Then, for $ X \in \dDN $, we have that 
\[
\| f^{ [ l ] } ( \frac{1}{m } X  ) \| \leq 
\sum_{ w \in \cFml } \frac{1 }{ m^l } \sup_{ w \in \cFml }  \left(  | f_w | \cdot \| X^w \| \right) 
\leq \sup_{ w \in \cFml } | f_w | ,
\]
therefore, for $ r \in ( 0, \frac{1}{m } ) $,
\[
\int_{ \dDN } \frac{1}{ N }
 \Tr \left(   f^{ [ l ] } ( r X )^\ast f^{ [ l ] } ( r X )   \right) d\mu_N 
\leq \sup_{ w \in \cFml } | f_w |^2.
\]
hence, expansion (\ref{eq:l2}) and  Corollary \ref{cor:24}(ii) give that
\begin{align*}
\int_{ \dDN } \frac{ 1 } { N } \Tr \left( f ( rX ) ^\ast \right. & \left. f ( r X ) \right) d \mu_N \\
= \sum_{ l = 0 }^\infty
&
 \int_{ \dDN } \frac{1}{ N }
 \Tr \left(   f^{ [ l ] } ( r X )^\ast f^{ [ l ] } ( r X )   \right) d\mu_N 
\leq \| f \|^2_{ l^2 ( \cF_m ) }.
\end{align*}

Therefore, using Corollary \ref{cor:24}(i), we have that for $ r \in ( 0, \frac{1 }{ m } ) $,
\[
\Phi(r ) = \sum_{ l = 0 }^\infty r^{2 l } ( \sum_{ w \in \cFml } | f_w |^2 )
\]
and $ \Phi $ extends analytically to $ (0, 1 ) $ and continously to $ [ 0, 1 ] .$

The proof for $ \Psi $ is similar, using Remark \ref{remark:3.10} and Corollary \ref{cor:2}.

For the converse, suppose that there  exists $ \delta   > 0 $ such that $ \Phi ( r ) $ exists for $ r < \delta $
and extends analitically to $ ( 0 , 1 ) . $  
In particular there  exists some $ N_0 $  such that the integral from the definition of $ \Phi (\cdot ) $ 
is finite if $ N > N_0 . $  Fix now $ N > N_0 $; equation (\ref{eq:nc1}) gives that there exists some $ \alpha > 0 $
such that the series 
$ \sum_{ w \in \cF_m } f_w X^w $
converges absolutely for $ X \in \alpha \mathbb{D}^m , $ 
 particularly  
$ \{ ( \frac{ \alpha}{ m } )^{  | w | } f_w \}_{ w \in \cF_m  } \in l^1 ( \cF_m )\subset l^2 ( \cF_m ) .$

Let $ R = \min \{ \delta, \frac{\alpha}{m } \} .$ Then Corollary \ref{cor:24} gives  that, for $ r \in ( 0, R ) $,
\[
\Phi(r) = \sum_{ l =0 }^\infty r^{ 2l } ( \sum_{ w \in \cFml } | f_w |^2 )
\]
and the conclusion follows since $ \Phi ( \cdot ) $ extends analytically to $ ( 0, 1 ) $.

As before, the proof for $ \Psi ( \cdot ) $ is similar, using equation (\ref{eq:nc1}) and Corollary \ref{cor:2}.

\end{proof}

\begin{prop}\label{oldprop:3.15}
For $ f \in \overline{ H^2_{1, m } } $ and $ g \in \overline { H^2_{ m, m } } $,  respectively  $ Y \in \Upsilon^m_1 \cap \mathbb{C}^{ M \times M } $ and $ Y^\prime \in \Upsilon^m_m \cap \mathbb{C}^{ M \times M } $, we have that
\begin{align*}
\varphi_{ f, Y } ( r ) &  
= \lim_{ N \lra \infty } \int_{ \dDN } \frac{1 }{  N  } \Tr \otimes \text{Id}_{ \mathbb{C}^{ M \times M } }
\left(
f( r X ) K_1 ( r X, Y^\ast)^\ast 
\right)
d\mu_N(X)\\
\psi_{ g, Y^\prime } ( r ) &
=
\lim_{ N \lra \infty } \int_{ \dBN } \frac{ 1 }{ N } \Tr\otimes \text{Id}_{ \mathbb{C}^{ M \times M } } 
\left( 
g (r X ) K_m ( rX, ( Y^\prime )^\ast )^\ast 
\right)
d\nu_N(X)
\end{align*}
are analytic functions of $ r $ for $ r $ small and they extend analytically to $ ( 0, 1 ) $ and continously to $ [ 0, 1 ] $.
Moreover, 
$ \displaystyle \lim_{ r \lra 1^{ - } }  \varphi_{ f, Y } ( r ) = f( Y ) $ 
and
$ \displaystyle \lim_{ r \lra 1^{ - } } \psi_{ g, Y^\prime } ( r ) = g( Y^\prime ) .$
\end{prop}

\begin{proof}
Let $ p \geq 1 $ and $ r \in ( 0, ( 2 m^2 p )^{ - 1 } ) ,$
let $ \{ f_w \}_{ w \in \cF_m } \in l^2_p ( \cF_m ) ,$
and consider
$ X \in \mathbb{C}^{ N \times N } $, $ Y \in \mathbb{C}^{ M \times M } $
 such that 
$\displaystyle \sup_{ w \in \cF_m } \|  Z^w \| = m(Z) < \infty $ 
and 
$ \displaystyle \sup_{ w \in \cF_m } \| X^w \| \leq 1 $.

First, note that the series 
$ \displaystyle \sum_{ l = 0 }^\infty ( \sum_{ w \in \cFml } r^{ 2 | w | } f_w p^{ | w | } Y^w ) $
converges asolutely, since
$ \displaystyle \| \{ ( rp)^{ 2 | w | } \}_{ w  } \|_{ l^2 ( \cF_m ) } < 2 $ 
and
\begin{align*}
\sum_{ l = 0 }^\infty 
( \sum_{ w \in \cFml } 
\| r^{ 2 | w | } f_w p^{ | w | } Y^w \| 
)
&\leq
m( Y ) \cdot \sum_{ l = 0 }^\infty 
( \sum_{ w \in \cFml } 
 ( r p )^{ 2 | w | } \cdot | p^{ - | w | } f_ w | \\
\leq
&
m ( Y ) \cdot \| \{ f_w \|_{ w  } \|^{ \frac{1}{2} }_{ l^2_p ( \cF_m ) } \cdot 
\| \{ ( rp)^{ 2 | w | } \}_{ w  } \|_{ l^2 ( \cF_m ) }^{ \frac{ 1 }{ 2 } } 
\end{align*}

Also, we have that
$ \displaystyle \| \{ ( r m p )^{ | w | } \}_{ w } \|_{ l^2( \cF_m ) } < 2 $
and
\begin{align*}
[
\sum_{ k = 0 }^\infty
(
\sum_{ w \in \cF_m^{ [ k ] } }
&
\| f_w \cdot r^{ | w | } X^w \| 
) ]
\cdot 
[
\sum_{ l = 0 }^\infty
(
\sum_{ v \in \cFml }
p^l \cdot r^l \| X^v \otimes (Y^v)^\ast \|
) ]\\
& \leq
[
\sum_{ k = 0 }^\infty 
m^k r^k 
(
\sum_{ w \in \cF_m^{ [ k ] } } | f_w | 
)]
\cdot
(
\sum_{ l = 0 }^\infty
p^l r^l m^l
\cdot m( Y )
)\\
&
\leq
[
\sum_{ k = 0 }^\infty
( r\cdot mp )^k 
\cdot
(
\sum_{ w \in \cF_m^{ [ k ] } }
 p^{ - | w | } | f_w | 
)]
\cdot 2 m(Y)\\
&
\leq
\| \{ f_w \}_{ w } \|^{ \frac{ 1 }{ 2 } }_{ l^2_p ( \cF_m ) }
\cdot
\| \{ ( r m p )^{ | w | } \}_{ w } \|^{ \frac{ 1 }{ 2 } }_{ l^2 ( \cF_m ) } \cdot 2 m(  Y ).
\end{align*}

Therefore, for $ p = 1 ,$  if $ f $ and $ Y $ are as in the statement of \ref{oldprop:3.15}, we have that
\begin{align*}
 &\int_{ \dDN } \frac{1 }{  N  } \Tr \otimes \text{Id}_{ \mathbb{C}^{ M \times M } }
\left(
f( r X ) K_1 ( r X, Y^\ast)^\ast 
\right)
d\mu_N(X)\\
 &\hspace{ 0.5 cm } = 
\int_{ \dDN } \frac{1 }{  N  } \Tr \otimes \text{Id}_{ \mathbb{C}^{ M \times M } }
\LARGE(
\sum_{ k, l = 0 }^\infty \sum_{ w \in \cF_m^{ [ k ] } } \sum_{ v \in \cFml }
f_w \cdot r^k X^w \cdot (X^v )^\ast \otimes Y^v 
\LARGE)
d\mu_N(X).
\end{align*}
Since $ \frac{1}{N}\Tr\otimes \text{Id}_{ \mathbb{C}^{M \times M } } $ is a bounded linear map, using 
Corollary \ref{cor:24}, the right hand side of the equation above equals
$ \displaystyle
\sum_{ l = 0 }^\infty ( \sum_{ w \in \cFml } f_w Y^w ). $
and the conclusion for $ \varphi_{ Y, f } $ follows.

The proof for $ \psi_{ g, Y^\prime } $ is analogous letting $ p = m $ and using Corollary \ref{cor:2}.

\end{proof}
\begin{defn}
For $ \Omega $ either $ \mathbb{B}^m $ or $ \mathbb{D}^m $, define
\begin{align*}
H^\infty (\Omega_{\text{nc}}  )
&=
\{ f : \Omega_{\text{nc}}  \lra \mathbb{C}_{\text{nc}} :  f \text{is a non-commutative function} 
 \text{ and }
\sup_{ Z \in \Omega }\| f ( Z ) \| < \infty 
\}\\
H^\infty
( \mathcal{B}_{\Omega } )
& = 
\{ f : \mathcal{B}_\Omega\lra \mathbb{C}_{\text{nc}} :  f \text{is a non-commutative function} 
 \text{ and }
\sup_{ Z \in \mathcal{B}_\Omega }\| f ( Z ) \| < \infty 
\}.
\end{align*}
\end{defn}

Obviously,  $ H^\infty (\Omega_{\text{nc}}  ) \subset H^\infty ( \mathcal{B}_{\Omega } ) $, 
since
$ \mathcal{B}_\Omega \subset \Omega_{\text{nc}}  $. We will further detail this inclusion below.
\begin{defn}
For $ \Omega $ either $ \mathbb{B}^m $ or $ \mathbb{D}^m $, define
\begin{align*}
\mathcal{M} ( \Omega_{\text{nc}}  )
& =
\{ f : \Omega_{\text{nc}}  \lra \mathbb{C}_{\text{nc}} : \   f \text{ is a non-commutative function which is also a  }\\  
& \text{bounded left multiplier for } H^2( \Omega_{\text{nc}} ) \}\\
\mathcal{M} ( \mathcal{B}_ \Omega  )
& =
\{ f : \mathcal{B}_ \Omega   \lra \mathbb{C}_{\text{nc}} : \   f \text{ is a non-commutative function which is also a  }\\  
& \text{bounded left multiplier for }  \overline{ H^2_{1, m } }  ,
 \text{if }   \Omega = \mathbb{B}^m ,   \text{ respectively }   \overline { H^2_{ m, m } }  
\text{ if }  \Omega = \mathbb{D}^m \}.
\end{align*}
where the multiplier norms are the natural ones.
\end{defn}

\begin{prop}
With the notations above, we have that
\[
H^\infty ( \Omega_{\text{nc} } ) \subseteq \mathcal{M} ( \Omega_{\text{nc} } ) 
\subseteq \mathcal{M} ( \mathcal{B}_\Omega ) \subseteq H^\infty ( \mathcal{B}_\Omega ) .
\]
\end{prop}

\begin{proof}
From the consideration above, we only need to prove the last inclusion. 
Consider $ g \in \mathcal{M} ( \mathcal{B}_\Omega ) $, denote $ M_g $ the left multiplier with  $ g $ 
 and take 
$ X \in  \mathcal{B}_\Omega \cap \mathbb{C}^{ M \times M }$, 
$  Y \in \Upsilon_{p}^m \cap \mathbb{C}^{ N \times N }$ .
From Poposition \ref{prop:11}, for any $ e_1, e_2 \in \mathbb{C}^M $ and 
$ f_1, f_2 \in \mathbb{C}^N $, we have that
\begin{align*}
\langle (M_g)^\ast e_1^\ast K( \cdot, X ) e_2 , &
f_1^\ast K( \cdot,  Y ) f_2 \rangle 
= 
\langle 
 e_1^\ast K( \cdot, X ) e_2 ,
M_g f_1^\ast K( \cdot,  Y ) f_2
\rangle \\
=&
\langle
g(\cdot) f_1^\ast K( \cdot,  Y ) f_2 , 
e_1^\ast K( \cdot, X ) e_2
\rangle^\ast\\
=&
(e_2^\ast g(X) f_1^\ast K( X,  Y ) f_2 e_1)^\ast,
\end{align*}
hence
$ (M_g)^\ast K( \cdot, X ) = K ( \cdot, X ) g(X)^\ast $ 
and since $ \| (M_g)^\ast K( \cdot, X ) \| \leq \| M_g \| \cdot \| K( \cdot, X )  \| $ and 
$ K ( \cdot, \cdot ) $ is a reproducing kernel, it follows that 
$ \| g(X) \| \leq \| M_g \| $.
\end{proof}


\bibliographystyle{alpha}


\end{document}